\documentclass{llncs}
\usepackage{epsfig}
\usepackage{amsmath}
\usepackage{amsfonts}
\usepackage{amssymb}

\def\X{{\bf X}}

\newtheorem{algorithm}[theorem]{Algorithm}

\newcommand{\OT}{\mathcal O}

\def\SymbReg{\textsuperscript{\textregistered}}

\begin{document}

\pagestyle{headings}

\title{Simulations of Multiscale Schr\"odinger Equations with Multiscale Splitting Approaches: Theory and Application}
\author{J\"urgen Geiser\inst{1} \and Amirbahador Nasari\inst{2}}
\institute{Ruhr University of Bochum, \\
Department of Electrical Engineering and Information Technology, \\
Universit\"atsstrasse 150, D-44801 Bochum, Germany \\
\email{juergen.geiser@ruhr-uni-bochum.de}
\and
Ruhr University of Bochum, \\
Department of Civil and Environmental Engineering, \\
Universit\"atsstrasse 150, D-44801 Bochum, Germany \\
\email{amirbahador.nasari@ruhr-uni-bochum.de}
}

\maketitle

\begin{abstract}

In this paper we present a novel multiscale splitting approach to solve 
multiscale Schr\"odinger equation, which have large different time-scales.
The energy potential is based on highly oscillating functions,
which are magnitudes faster than the transport term.
We obtain a multiscale problem and a highly stiff problem, while
standard solvers need to small time-steps.
We propose multiscale solvers, which are based on operator 
splitting methods and we decouple the diffusion and reaction part of the 
Schr\"odinger equation.
Such a decomposition allows to apply a large time step 
for the implicit time-discretization of the diffusion part and
small time steps for the explicit and highly oscillating reaction part. 
With extrapolation steps, we could reduce the computational time
in the highly-oscillating time-scale, while we relax into the 
slow time-scale.
We present the numerical analysis of the extrapolated operator
splitting method.
First numerical experiments verified the benefit of the
extrapolated splitting approaches.

\end{abstract}



\section{Introduction}

The Schr\"odinger Equation is wel-known for modelling the basis 
of quantum mechanics. 
The state of a particle is described by its wave-function $ \Psi$,
which is a function of space (position) $x$ and time $t$.
The wave-function is given as a complex variable and it is delicate to
attribute any distinct physical meaning to such complex notation.
We consider a delicate multiscale problem based on 
different time-scales in the diffusion and reaction part of the
time-dependent Schr\"odinger equation, see \cite{tak2008}, \cite{griffiths2004} and \cite{singh2016}.
Based on the large scale-dependencies, we have to apply 
time-restrictions to the numerical methods, see \cite{hai96},
such that we apply multiscale methods to decompose the different scales
and accelerate the solver process. We propose novel multiscale methods,
which are based on AB-splitting and ABA-splitting methods,
see \cite{fargei05} and \cite{geiser_2009_1}, and additional modification with extrapolation 
and averaging ideas. We also test heterogeneous multiscale methods
(HMM), see \cite{we_2007} and apply the ideas of the equation-free methods (EFM), see \cite{kev2009}.
We compare the different novel splitting approaches with the standard splitting schemes, see also \cite{geiser2018}.
Such combinations allow to overcome the delicate stiffness problems of the
scale-dependent Schr\"odinger equation and reduce the computational time.

The paper is outlined as following. The model is introduced in Section \ref{modell}.
In Section \ref{methods}, we discuss the different numerical methods and present the
convergence analysis. The numerical experiments are done in Section \ref{numeric} and
the conclusion is presented in Section \ref{concl}.

\section{Mathematical Model}
\label{modell}

In the following, we deal with the multiscale Schr\"odinger equation,
which is given for the first dimension as:
\begin{eqnarray}
\label{schroed_1}
&& i \; \epsilon \frac{\partial \Psi}{\partial t} = - \epsilon^2 \frac{\partial^2}{\partial x^2} \Psi +  V(x) \Psi , \\
\label{schroed_2}
&& \frac{\partial \Psi}{\partial t} = i \; \epsilon \frac{\partial^2}{\partial x^2} \Psi - i \; \frac{1}{\epsilon}  V(x) \Psi ,
\end{eqnarray}
where $\epsilon \in (0,1)$, $V(x)$ is the potential, e.g., $x^2$ or $\sin(2 \pi x)$ or $\exp(\sin(2 \pi x))$.

He, we assume the connection to the instationary Schr\"odinger equation,
which is given if we apply $\hbar \rightarrow \epsilon$ and $\frac{\hbar^2}{m} \rightarrow \epsilon^2$.

For the different numerical schemes, we deal with two equations:

\begin{itemize}
\item Scaled equation (test example) \\
In the following, we apply a scaled Schr\"odinger equation with $\epsilon$,
where we scale $L = 1.0$, therefore $\epsilon \in [10^{-6}, 10^{-3}]$.
\begin{align}
\label{multi_schroed_1}
i \epsilon \frac{\partial \Psi(x,t)}{\partial t} &= - \epsilon^2 \frac{\partial^2 \Psi(x,t)}{\partial x^2} + V(x) \Psi(x,t) , \\
\label{multi_schroed_2}
  \frac{\partial \Psi(x,t)}{\partial t} &= i \; \epsilon \frac{\partial^2 \Psi(x,t)}{\partial x^2} - \frac{i}{\epsilon} V(x) \Psi(x,t),
\end{align}
where the initial condition is given as $\Psi(x, 0) = sin(\pi x)$ with $x \in [0, 1]$ and $t \in [0, 1]$.

\item Unscaled equation (full Schr\"odinger equation)
\begin{eqnarray}
\label{instat_diff_1}
&& i \hbar \frac{\partial}{\partial t} \Psi(x,t) =  \hat{H} \Psi(x,t), \\
&& \frac{\partial}{\partial t} \Psi(x,t) = - i \; \frac{1}{\hbar}  \hat{H} \Psi(x,t)  ,
\end{eqnarray}
where $\hat{H} = - \frac{\hbar^2}{2 \mu} \nabla^2 + V(x,t)$. \\
We apply the initial conditions: \\
$\Psi(x, 0) = \exp(-0.5 (\frac{x- x_c}{s})^2)  cos(\frac{2 \pi (x-x_c)}{\lambda}) + i  \exp(-0.5 (\frac{x- x_c}{s})^2)  sin(\frac{2 \pi (x-x_c)}{\lambda})$, where $x \in [0, L]$, where we have $L= 4 \; 10^{-9} \; [m]$, $\lambda = L / 40$ and $s = L/25$.

Further, we assume $V(x,t)=0$ and $V(x,t) = e \;x^2$, \\
where $e =1.6021766 \; 10^{-19} \; [C]$ is the charge of the electron.

\end{itemize}

\begin{remark}

The instationary Schr\"odinger equation is given as:
\begin{eqnarray}
\label{instat_diff_1}
&& i \hbar \frac{\partial}{\partial t} \Psi(x,t) =  \hat{H} \Psi(x,t), \\
&& \frac{\partial}{\partial t} \Psi(x,t) = - i \; \frac{1}{\hbar}  \hat{H} \Psi(x,t)  ,
\end{eqnarray}
where $\hat{H} = - \frac{\hbar^2}{2 \mu} \nabla^2 + V(x,t)$. further $\mu$ is the particle's reduced mass, $i$ is imaginary unit, V(x,t) is its potential energy, $\nabla^2$ is the Laplacian (a differential operator), 
and $\Psi$ is the wave function.
The instationary Schr\"odinger equation can be rewritten with respect to an $\epsilon$-parameter to a multiscale Schr\"odinger equation, see Equation (\ref{schroed_1}) and (\ref{schroed_2}).

\end{remark}

\section{Numerical methods}
\label{methods}

In the following, we present the different splitting 
methods, which are discussed as:
\begin{itemize}
\item Standard operator splitting methods:
\begin{itemize}
\item AB-splitting method or Lie-Trotter splitting method, see \cite{trotter59},
which solve each operator in a separate equation and apply different time-steps for each operator.
\item ABA-splitting or  BAB-splitting method or Strang splitting method, see \cite{stra68}, which improves the AB-splitting method with respect to additional
steps for the separated operators. Further, we also assume to deal with the different time-steps for each operator.
\end{itemize}
\item Modified multiscale splitting methods.
\begin{itemize}
\item HMM-AB-splitting method, see \cite{geiser2018},
which combines the Heterogeneous Multiscale method and the AB-splitting method, such that we could embed the
microscopic operator equation with a smaller amount of microscopic time-steps into the macroscopic operator equation.
\item Extrapolated AB-splitting method, see the ideas in  \cite{kev2009}, which combines the equation free methods and the AB-splitting method.
Therefore, we could reduce the large amount of microscopic time-steps of the microscopic operator equation and apply extrapolation methods and
embed the results into the macroscopic operator equation.
\item Higher-Order Extrapolated ABA-splitting method, see the ideas in  \cite{kev2009}, here we combine higher order
extrapolation methods and ABA-splitting method.
Therefore, we could reduce the large amount of microscopic time-steps of the microscopic operator equation and apply extrapolation methods and
embed the results into the macroscopic operator equation. Further, we also obtain a second order scheme.
\end{itemize}
\end{itemize}

\subsection{Instationary Schr\"odinger equation}

We discretize the instationary Schr\"odinger equation (\ref{instat_diff_1})
with implicit time-discretization and second order
spatial discretization methods.

The discretization is given as:
\begin{align}
\label{diff_1}
 i \hbar \frac{\Psi(x,t)-\Psi(x,t-\Delta t)}{\Delta t} & = - \frac{\hbar^2}{2 m} \frac{\Psi(x+\Delta x,t) - 2\Psi(x,t) + \Psi(x-\Delta x,t)}{\Delta x^2} + \\
& + V(x,t) \Psi(x,t), \nonumber
\end{align}
where we obtain:
\begin{align}
\label{schroed_1_4}
 \frac{\Psi(x,t)-\Psi(x,t-\Delta t)}{\Delta t} &= i \frac{\hbar}{2 m} \frac{\Psi(x+\Delta x,t) - 2\Psi(x,t) + \Psi(x-\Delta x,t)}{\Delta x^2} - \\
& - \frac{i}{\hbar} V(x,t) \Psi(x,t) . \nonumber
\end{align}

Then, we simplify the discretized Equation (\ref{schroed_1_4}) and obtain:
\begin{align}
\label{schroed_1_5}
\Psi(x,t-\Delta t) &=  - \frac{i \; \hbar \Delta t}{ 2 m \Delta x^2 } (\Psi(x+\Delta x,t) - 2\Psi(x,t) + \\ 
+ \Psi(x-\Delta x,t))+(1 + \frac{i \; \Delta t V(x,t)}{\hbar})\Psi(x,t) . \nonumber
\end{align}

Then, we rewrite the Equation (\ref{schroed_1_5}) in a matrix-notation, which can be programmed to software-package, e.g., MATLAB\SymbReg:
\begin{align}
\label{diff_1}
& \Psi(k-1) =  - \frac{i\; \hbar \Delta t}{2 m \Delta x^2 } \begin{bmatrix}
    -2 & 1 &   0 & 0 & \dots  & 0 \\
     1 & -2 &  1 & 0 & \dots  & 0 \\
     0 &  1 & -2 & 1 & \dots  & 0 \\
    \vdots & \vdots & \vdots & \vdots & \ddots & \vdots \\
    0 & 0 & 0 & 0 & \dots  & -2
\end{bmatrix}
\Psi(k)  \\
& \hspace{-0.5cm} + \begin{bmatrix}
    (1 + \frac{i \; \Delta t V(x_0, t)}{\hbar}) & 0 &   0 & 0 & \dots  & 0 \\
     0 & (1+\frac{i \; \Delta t V(x_2, t)}{\hbar}) &  0 & 0 & \dots  & 0 \\
     0 &  0 & (1+\frac{i \; \Delta t V(x_3, t)}{\hbar}) & 0 & \dots  & 0 \\
    \vdots & \vdots & \vdots & \vdots & \ddots & \vdots \\
    0 & 0 & 0 & 0 & \dots  & (1 + \frac{i \; \Delta t V(x_{M}, t)}{\hbar})
\end{bmatrix}  \Psi(k) , \nonumber
\end{align}
where $j= 0, 1, \ldots, J$ are the spatial points with $\Delta x = 1/J$ and $x_j = j \; \Delta x$.
Then, we obtain the operator equation:
\begin{align}
\label{diff_1}
\Psi(k-1)=\textbf{B}\Psi(k) , \\
\Psi(k)=\mathbf{B}^{-1}\Psi(k-1) .
\end{align}

\begin{remark}

The standard finite difference discretization with implicit \\
time-discretization
and second order spatial discretization leads to large linear equation systems.
Here, we do not separate the operators with respect to their different
time-steps, such that we solve a stiff equation system and neglect the
microscopic operator, see \cite{butcher2008} and \cite{geiser_2016}.
\end{remark}

\subsection{AB and ABA Splitting for the Multiscale Schr\"odinger equation}

We apply the following idea AB splitting method, while the diffusion operator is implicit discretized in time (large time steps) and the
potential operator (reaction operator) is explicit discretized in time.

We apply the following discretization in time and space with the point $(x,t)$ as:
\begin{eqnarray}
\label{diff_1}
&& (1 - \frac{i \; \Delta t V(x)}{\epsilon})\Psi(x,t - \Delta t) = \nonumber \\
&& = \Psi(x,t) - \frac{i \; \epsilon \Delta t}{\Delta x^2} (\Psi(x+\Delta x,t) - 2\Psi(x,t) + \Psi(x-\Delta x,t)).
\end{eqnarray}

Then, we have the semi-discretized operator equation as:
\begin{align}
(I - \Delta t\textbf{A}_1) \Psi(k-1) &  = (I + \Delta t \textbf{A}_2) \Psi(k) , \\
\textbf{B}_1 \Psi(k-1) &  = \textbf{B}_2 \Psi(k) , 
\end{align}
where the potential matrix is $\textbf{B}_1 = \begin{bmatrix}
    1 - \tilde{a}_0 & 0 &   0 & 0 & \dots  & 0 \\
    0 & 1 - \tilde{a}_1 & 0 & 0 & \dots  & 0 \\
     0 &  0 & 1 - \tilde{a}_2 & 0 & \dots  & 0 \\
    \vdots & \vdots & \vdots & \vdots & \ddots & \vdots \\
    0 & 0 & 0 & 0 & \dots  & 1 -\tilde{a}_M
\end{bmatrix} , $

with $\tilde{a}_j = \frac{i \; \Delta t V(x_j)}{\epsilon}$ and \\
and the diffusion matrix is 
$\textbf{B}_2 = \begin{bmatrix}
    1 + 2 a & - a &   0 & 0 & \dots  & 0 \\
    - a & 1 + 2 a &  - a & 0 & \dots  & 0 \\
     0 &  - a & 1 + 2 a & -a & \dots  & 0 \\
    \vdots & \vdots & \vdots & \vdots & \ddots & \vdots \\
    0 & 0 & 0 & 0 & \dots  & 1 + 2 a
\end{bmatrix} , $ 

with $a =  \frac{i \; \epsilon \Delta t}{\Delta x^2 }$.

We have the AB-splitting method as:
\begin{align}
\label{diff_1}
& u_1(k) = \textbf{B}_1 \Psi(k-1) , \\
& \Psi(k) = (\textbf{B}_2)^{-1} u_1(k) . 
\end{align}

Based on the small scale of operator $\textbf{B}_1$, we improve as:
\begin{align}
\label{diff_1}
& u_1(k-1,m+1) = \tilde{\textbf{B}}_1 \Psi(k-1,m) ,
\end{align}
with $\tilde{\textbf{B}}_1 = \begin{bmatrix}
    1 - \tilde{\tilde{a}}_0 & 0 &   0 & 0 & \dots  & 0 \\
    0 & 1 - \tilde{\tilde{a}}_1 & 0 & 0 & \dots  & 0 \\
     0 &  0 & 1 - \tilde{\tilde{a}}_2 & 0 & \dots  & 0 \\
    \vdots & \vdots & \vdots & \vdots & \ddots & \vdots \\
    0 & 0 & 0 & 0 & \dots  & 1 - \tilde{\tilde{a}}_M
\end{bmatrix} , $

with $\tilde{\tilde{a}}_j = \frac{i \; \delta t V(x_j)}{\epsilon}$.

Further, we have $m = 0, \ldots, M-1$ and $\delta t = \frac{\Delta t}{M}$ with $M = \frac{1}{\epsilon}$ and $M$ is assumed to be an integer. We have $\Psi(k-1,0) = \Psi(k-1)$ as the initialization and  $\Psi(k) = \Psi(k-1, M)$ as the result.

The algorithm is given as:
\begin{align}
\label{diff_1}
& u_1(k) = (\tilde{\textbf{B}}_1)^M \Psi(k-1) , \\
& \Psi(k) = (\textbf{B}_2)^{-1} u_1(k) ,
\end{align}
where $M$ is the number of intermediate time steps in the
microscopic scale means $\delta t = \Delta t / M$.

The AB splitting method is given in Algorithm \ref{algo_1} as:
\begin{algorithm}
\label{algo_1}

\begin{itemize}
\item Step 1: Solving the microscopic equation (potential part):
\begin{eqnarray}
& u_1(t^{n+1}) = (\tilde{\textbf{B}}_1)^M \Psi(t^n) , 
\end{eqnarray}
with $\tilde{\textbf{B}}_1$ is the discretized potential operator (small scale) and $M$ is the number of small time steps with $\delta t = \Delta t /M$.
\item Step 2: Solving the macroscopic equation (diffusion part)
\begin{eqnarray}
& \Psi(t^{n+1}) = (\textbf{B}_2)^{-1} u_1(t^{n+1}) ,
\end{eqnarray}
with $\textbf{B}_2$ is the discretized diffusion operator (large scale).

\end{itemize}

\end{algorithm}

\begin{remark}
For the ABA and BAB methods, we have socalled three step methods, 
while we apply for example the A operator with the $\Delta t/2$ timestep 
(at least 2 times, means at the step 1 and step 3) and the 
one time the $\Delta t$ timestep for B.
\end{remark}

The ABA splitting method is given in Algorithm \ref{algo_1_1} as:
\begin{algorithm}
\label{algo_1_1}

\begin{itemize}
\item Step 1: Solving the microscopic equation (potential part):
\begin{eqnarray}
& u_1(t^{n+1}) = (\tilde{\textbf{B}}_1)^M \Psi(t^n) , 
\end{eqnarray}
with $\tilde{\textbf{B}}_1$ is the discretized potential operator (small scale) and $M$ is the number of small time steps with $\delta t = (\Delta t/2)/ M$.
\item Step 2: Solving the macroscopic equation (diffusion part)
\begin{eqnarray}
& u_2(t^{n+1}) = (\textbf{B}_2)^{-1} u_1(t^{n+1}) ,
\end{eqnarray}
with $\textbf{B}_2$ is the discretized diffusion operator (large scale).

\item Step 3: Solving the microscopic equation (potential part):
\begin{eqnarray}
& \Psi(t^{n+1}) = (\tilde{\textbf{B}}_1)^M u_2(t^{n+1}) , 
\end{eqnarray}
with $\tilde{\textbf{B}}_1$ is the discretized potential operator (small scale) and $M$ is the number of small time steps with $\delta t = (\Delta t / 2)/ M$.
\end{itemize}

\end{algorithm}

The BAB splitting method is given in Algorithm \ref{algo_2_1},
where we have the first B-step with $(B_2)^{-1}$ and $\Delta t/2$,
then the  A-step with $(\tilde{B}_1)^M$ and time-step $\delta t = \Delta t/M$ 
and the second B-step with $(B_2)^{-1}$ with $\Delta t/2$.
\begin{algorithm}
\label{algo_2_1}

\begin{itemize}
\item Step 1: Solving the macroscopic equation (diffusion part) with
a half time step $\Delta t/2$
\begin{eqnarray}
& u_1(t^{n+1}) = (\textbf{B}_2)^{-1} \phi(t^{n}) ,
\end{eqnarray}
with $\textbf{B}_2$ is the discretized diffusion operator (large scale)
with the half time-step means $\Delta t/2$.

\item Step 2: Solving the microscopic equation (potential part):
\begin{eqnarray}
& u_2(t^{n+1}) = (\tilde{\textbf{B}}_1)^M u_1(t^{n+1}) , 
\end{eqnarray}
with $\tilde{\textbf{B}}_1$ is the discretized potential operator (small scale) and $M$ is the number of small time steps with $\delta t = \Delta t/ M$.

\item Step 3:  Solving the macroscopic equation (diffusion part) with a
next half time-step $\Delta t/2$
\begin{eqnarray}
& \phi(t^{n+1}) = (\textbf{B}_2)^{-1} u_2(t^{n+1}) ,
\end{eqnarray}
with $\textbf{B}_2$ is the discretized diffusion operator (large scale)
with the half time-step means $\Delta t/2$.

\end{itemize}

\end{algorithm}

\begin{remark}

Here, we have the stability condition with respect for
reaction term, which is give as:
\begin{eqnarray}
\delta t \le \frac{\epsilon}{V(x)},
\end{eqnarray}
where $V(x) = e \; x^2$ is a very small number, such that we only have a critical
time-step, if we assume $\epsilon \approx e$, then we have to deal with small time-steps.

\end{remark}

\begin{remark}
\label{rem_AB_ABA}

Based on the operator splitting methods, see also  \cite{fargei05} and \cite{geiser2011},
we have the following splitting errors:
\begin{itemize}
\item AB-splitting:
\begin{eqnarray}
err_{AB, local}(\Delta t) \le \Delta t^2 C_{AB},
\end{eqnarray}
where the constant $C_{AB}$ is depending of $\| (\textbf{A}_1 \|$, $\| (\textbf{A}_2 \|$, see \cite{geiser2011}.
\item AB-splitting:
\begin{eqnarray}
err_{ABA, local}(\Delta t) \le \Delta t^3 C_{ABA},
\end{eqnarray}
where the constant $C_{ABA}$ is depending of $\| (\textbf{A}_1 \|$, $\| (\textbf{A}_2 \|$, see \cite{geiser2011}.
\end{itemize}

Further $\| \cdot \|$ is an appropriated matrix norm in the appropriate Hilbert-space.
\end{remark}

\subsection{Modified Multiscale AB for the Multiscale Schr\"odinger equation}

We start with the discretization of the multiscale Schr\"odinger
equation (\ref{schroed_1}) and (\ref{schroed_2}).

We modify the small scale of operator $\textbf{B}_1$ as following:
\begin{align}
\label{diff_1}
& u_1(k-1,m+1) = \tilde{\textbf{B}}_1 \Psi(k-1,m) ,
\end{align}
where $m = 0, \ldots, M-1$ and $\delta t = \frac{\Delta t}{M}$ with $M = \frac{1}{\epsilon}$ and $M$ is assumed to be an integer. 
We have $\Psi(k-1,0) = \Psi(k-1)$ as the initialisation and $\Psi(k) = \Psi(k-1, m)$
with $m = 1, \ldots, \tilde{M}$ the successor results with $\tilde{M} \le M$.

We have the following multiscale AB algorithms:

\begin{enumerate}
\item Full-AB: \\
The algorithm is given as:
\begin{align}
\label{diff_1}
& u_1(k) = (\tilde{\textbf{B}}_1)^M \Psi(k-1) , \\
& \Psi(k) = (\textbf{B}_2)^{-1} u_1(k) ,
\end{align}
where $M$ is the number of intermediate time steps in the
microscopic scale means $\delta t = \Delta t / M$.
Here, we applied the full time-interval.

\item HMM-AB: \\
The algorithm is given as:
\begin{align}
\label{diff_1}
& u_1(k-1, m) = (\tilde{\textbf{B}}_1)^m \Psi(k-1) , \;  \mbox{for} \; m = 1, \ldots, \tilde{M} , \\
& u_1(k) = \frac{1}{\tilde{M}} \; \sum_{m=1}^{\tilde M} u_1(k-1, m) , \\
& \Psi(k) = (\textbf{B}_2)^{-1} u_1(k) ,
\end{align}
where $M$ is the number of intermediate time steps in the
microscopic scale means $\delta t = \Delta t / M$. \\

\item Extrapolated-AB: \\
The algorithm is given as:
\begin{align}
\label{extra_1_1}
& u_1(k-1, \tilde{M}-1) = (\tilde{\textbf{B}}_1)^{\tilde{M}-1} \Psi(k-1) , \\
& u_1(k-1, \tilde{M}) = (\tilde{\textbf{B}}_1)^{\tilde{M}} \Psi(k-1) , \\
& u_1(k) = u_1(k-1, \tilde{M} - 1) + \nonumber \\
& + (\Delta t - \delta t \; (\tilde{M}-1)) \frac{u_1(k-1, \tilde{M}) -  u_1(k-1, \tilde{M} - 1)}{\delta t} , \\
\label{extra_1_2}
& \Psi(k) = (\textbf{B}_2)^{-1} u_1(k) ,
\end{align}
where $M$ is the number of intermediate time steps in the
microscopic scale means $\delta t = \Delta t / M$.

\item Higher-order Extrapolated-AB: \\
The algorithm is given as:
\begin{align}
\label{extra_2_1}
& u_1(k-1, \tilde{M}-1) = (\tilde{\textbf{B}}_1)^{\tilde{M}-1} \Psi(k-1) , \\
& u_1(k-1, \tilde{M}-2) = (\tilde{\textbf{B}}_1)^{\tilde{M}-2} \Psi(k-1) , \\
& u_1(k-1, \tilde{M}) = (\tilde{\textbf{B}}_1)^{\tilde{M}} \Psi(k-1) , \\
& u_1(k) = u_1(k-1, \tilde{M} - 1) + \nonumber \\
& + (\Delta t - \delta t \; (\tilde{M}-1)) \frac{u_1(k-1, \tilde{M}) -  u_1(k-1, \tilde{M} - 1)}{\delta t} + \nonumber \\
& + \frac{(\Delta t - \delta t \; (\tilde{M}-1))^2}{2} \cdot \nonumber \\
& \cdot  \frac{(u_1(k-1, \tilde{M}) - 2 u_1(k-1, \tilde{M} - 1) + u_1(k-1, \tilde{M}-2) )}{(\delta t)^2} , \\
\label{extra_2_2}
& \Psi(k) = (\textbf{B}_2)^{-1} u_1(k) ,
\end{align}
where $M$ is the number of intermediate time steps in the
microscopic scale means $\delta t = \Delta t / M$.

\item Higher-order Extrapolated-ABA: \\
The algorithm is given as:
\begin{align}
\label{extra_2_1}
& u_1(k-1, \tilde{M}-1) = (\tilde{\textbf{B}}_1)^{\tilde{M}-1} \Psi(k-1) , \\
& u_1(k-1, \tilde{M}-2) = (\tilde{\textbf{B}}_1)^{\tilde{M}-2} \Psi(k-1) , \\
& u_1(k-1, \tilde{M}) = (\tilde{\textbf{B}}_1)^{\tilde{M}} \Psi(k-1) , \\
& u_1(k) = u_1(k-1, \tilde{M} - 1) + \nonumber \\
& + (\Delta t - \delta t \; (\tilde{M}-1)) \frac{u_1(k-1, \tilde{M}) -  u_1(k-1, \tilde{M} - 1)}{\delta t} + \nonumber \\
& + \frac{(\Delta t - \delta t \; (\tilde{M}-1))^2}{2} \cdot \nonumber \\
& \cdot \frac{(u_1(k-1, \tilde{M}) - 2 u_1(k-1, \tilde{M} - 1) + u_1(k-1, \tilde{M}-2) )}{(\delta t)^2} , \\
& u_2(k) = (\textbf{B}_2)^{-1} u_1(k) , \\
& u_3(k-1, \tilde{M}-1) = (\tilde{\textbf{B}}_1)^{\tilde{M}-1} u_2(k) , \\
& u_3(k-1, \tilde{M}-2) = (\tilde{\textbf{B}}_1)^{\tilde{M}-2} u_2(k) , \\
& u_3(k-1, \tilde{M}) = (\tilde{\textbf{B}}_1)^{\tilde{M}} u_2(k) , \\
& \label{extra_2_2}
\Psi(k) = u_3(k-1, \tilde{M} - 1) + \\
& + (\Delta t - \delta t \; (\tilde{M}-1)) \frac{u_3(k-1, \tilde{M}) -  u_3(k-1, \tilde{M} - 1)}{\delta t} + \nonumber \\
& + \frac{(\Delta t - \delta t \; (\tilde{M}-1))^2}{2} \cdot \nonumber \\
& \cdot \frac{(u_3(k-1, \tilde{M}) - 2 u_3(k-1, \tilde{M} - 1) + u_3(k-1, \tilde{M}-2) )}{(\delta t)^2} , \nonumber
\end{align}
where $M$ is the number of intermediate time steps in the
microscopic scale means $\delta t = (\Delta t / 2) / M$.

\end{enumerate}

In the following, we have the remark to the 
extrapolation methods, see Remark \ref{rem_extra_1}.

\begin{remark}
\label{rem_extra_1}

We assume to compute $u^{n+1}$ and apply the Taylor-expansion as following:
\begin{eqnarray}
\label{series}
u(t^{n+1}) = u(t^n) + \Delta t \frac{\partial u}{\partial t}|_{t^n} +  \frac{\Delta t^2}{2!} \frac{\partial^2 u}{\partial t^2}|_{t^n} + \frac{\Delta t^3}{3!} \frac{\partial^2 u}{\partial t^2}|_{t^n} + \ldots ,
\end{eqnarray}
where $\Delta t = t^{n+1} - t^n$.

We assume to deal with finer time-steps $\delta t = \Delta t / M$ and we have computed a series of finer resolutions $u(t^n, m)$ with $m=1, \ldots, \tilde{M}$
and the initialisation $u(t^n, 0) = u(t^n)$.

Then we can apply the extrapolation:
\begin{eqnarray}
\label{diff_1}
&& u(t^{n+1}) \approx u(t^n, \tilde{M} - 1) +  \\
&& + (\Delta t - \delta t \; (\tilde{M}-1)) \frac{u(t^n, \tilde{M}) -  u(t^n, \tilde{M} - 1)}{\delta t} , \nonumber \\
&& + \frac{(\Delta t - \delta t \; (\tilde{M}-1))^2}{2!} \; \frac{(u(t^n, \tilde{M}) - 2 u(t^n, \tilde{M} - 1) + u(t^n, \tilde{M}-2) )}{(\delta t)^2} + \OT(\delta t^3) . \nonumber
\end{eqnarray}
where we obtain an global error based on $\OT(\delta t^2)$.

\end{remark}

\subsection{Numerical Analysis of the extrapolated AB splitting method}

In the following, we analyze the consistence and order
of the extrapolated AB-splitting method (\ref{extra_1_1})--(\ref{extra_1_2}) and the higher extrapolated AB-splitting method
(\ref{extra_2_1})--(\ref{extra_2_2}).

We assume, we have linear bounded operators $A_1, A_2: \!  {\X} \rightarrow {\X} $, which are semi-discretized with finite difference schemes of the
multiscale Schr\"odinger equation given in (\ref{multi_schroed_1})-(\ref{multi_schroed_2}).

\smallskip
\begin{theorem}{\label{Th1}}
Let $A_1, A_2 \in {\mathcal L(X)} $ are given linear bounded operators
and we consider the abstract Cauchy problem

\begin{equation}\label{eq:ACP}
\begin{array}{c}
{\displaystyle \partial_t c(t) = A c(t) + B c(t), \quad 0 < t
\leq T } \\
\noalign{\vskip 1ex} {\displaystyle c(0)=c_0. }
\end{array}
\end{equation}
\noindent
 Then the  problem (\ref{eq:ACP}) has a unique solution.

We apply the extrapolated AB-splitting method (\ref{extra_1_1})--(\ref{extra_1_2}) and the higher extrapolated AB-splitting method
(\ref{extra_2_1})--(\ref{extra_2_2}) with uniform macroscopic time-steps $\Delta t = t^{n+1} - t^n$ with $n = 0, \ldots, N$.
The time-interval is $[0, T]$ and we have $t^{N+1} = T$ and $t^0 = 0$.

Then, we have the following convergence results:
\begin{enumerate}
\item Extrapolated AB splitting method: The local convergence 
is given as $(C_{AB} + C_{extrapol}) \delta t^2$, where $C_{AB}$ and
$C_{extrapol}$ are bounded constants.
The global convergence of the extrapolated AB splitting method 
is given as $\OT(\delta t)$ and $\Delta t = M \delta t$ is the microscopic time-step.
\item Higher Extrapolated ABA splitting method: The local convergence 
is given as $(C_{ABA} + C_{high extra}) \delta t^3$, where $C_{ABA}$ and
$C_{high extra}$ are bounded constants.
The global convergence of the extrapolated AB splitting method 
is given as $\OT(\delta t^2)$ and $\Delta t /2 = M \delta t$ is the microscopic time-step.
\end{enumerate}

\end{theorem}

\begin{proof}

\begin{enumerate}
\item Convergence of the extrapolated AB splitting method: \\
We have the following local error function $E(\Delta t, \delta t)= u(t) - u_{extrapol AB}(t)$, where $\Delta t = t^{n+1} - t^n$ and $\delta t = \Delta t /M$ we have the relations
\begin{eqnarray}
&& || u(t) - u_{extrapol AB}(t) || \le \\
&& \le || u(t) -  u_{AB}(t) + u_{AB}(t) - u_{extrapol AB}(t) || , \nonumber
\end{eqnarray}
we know, that the splitting error of an AB-splitting is given in
Remark (\ref{rem_AB_ABA}) with:
\begin{eqnarray}
|| u(t^{n+1}) -  u_{AB}(t^{n+1}) || \le C_{AB} \Delta t^2 ||u(t^n)||. 
\end{eqnarray}
Further, the error of the first order extrapolation method is given as, see Remark \ref{rem_extra_1}
\begin{eqnarray}
|| u_{AB}(t^{n+1}) - u_{extrapol AB}(t^{n+1}) || \le C_{extrapol} \delta t^2 ||u(t^n)||.
\end{eqnarray}
Then, we obtain:
\begin{eqnarray}
|| u(t^{n+1}) - u_{extrapol AB}(t^{n+1}) || && \le  C_{AB} \Delta t^2 + C_{extrapol}) \delta t^2 \nonumber \\
&& \le (C_{AB} M^2 + C_{extrapol}) \delta t^2 ||u(t^n)||.
\end{eqnarray}
Further the global error is given as:
\begin{eqnarray}
E_{global}(T) && = (N+1) M ||E(\Delta t, \delta t)|| = \\
&& = (N+1) M \delta t || \frac{ E(\Delta t, \delta t)}{\delta t} || \le \OT(\delta t). \nonumber
\end{eqnarray}

\item Convergence of the higher extrapolated ABA splitting method: \\
We have the following local error function $E_{ABA}(\Delta t, \delta t)= u(t) - u_{extrapol ABA}(t)$, where $\Delta t = t^{n+1} - t^n$ and $\delta t = (\Delta t /2) /M$ we have the relations
\begin{eqnarray}
&& || u(t) - u_{extrapol ABA}(t) || \le \\
&& \le  || u(t) -  u_{ABA}(t) + u_{ABA}(t) - u_{extrapol ABA}(t) || , \nonumber
\end{eqnarray}
we know, that the splitting error of an AB-splitting is given in
Remark (\ref{rem_AB_ABA}) with:
\begin{eqnarray}
|| u(t^{n+1}) -  u_{ABA}(t^{n+1}) || \le C_{ABA} \Delta t^3 ||u(t^n)||. 
\end{eqnarray}
Further, the error of the second order extrapolation method is given as, see Remark \ref{rem_extra_1}
\begin{eqnarray}
|| u_{ABA}(t^{n+1}) - u_{extrapol ABA}(t^{n+1}) || \le C_{higher extra}) \delta t^3 ||u(t^n)||.
\end{eqnarray}
Then, we obtain:
\begin{eqnarray}
&& || u(t^{n+1}) - u_{extrapol ABA}(t^{n+1}) || \le  C_{ABA} \Delta t^2 + C_{higher extra}) \delta t^3 \nonumber \\
&& \le (C_{ABA} (2 M)^2 + C_{extrapol}) \delta t^3 ||u(t^n)||.
\end{eqnarray}
Further the global error is given as:
\begin{eqnarray}
E_{global}(T) && = (N+1) 2 M ||E_{ABA}(\Delta t, \delta t)|| = \\
&& = (N+1) 2 M \delta t || \frac{ E_{ABA}(\Delta t, \delta t)}{\delta t} || \le \OT(\delta t^2). \nonumber
\end{eqnarray}

\end{enumerate}

\end{proof}

\begin{remark}
If we apply of higher order splitting schemes, e.g., ABA splitting method,
for the multiscale solver, we also need higher order extrapolation schemes.
Therefore, we have taken into account higher order extrapolation schemes
for improved splitting approaches.
\end{remark}

\section{Numerical Experiments}
\label{numeric}

In the following experiment, we test the multiscale solver methods for the one-dimensional timedependent Schr\"odinger equation with
a highly oscillating potential. \\

\subsection{Test example 1}

For the instationary Schr\"odinger equation, which is 
given in Equation (\ref{schroed_1}), we test the following methods:
\begin{itemize}
\item Finite difference scheme for multiscale Schr\"odinger equation (unsplitted),
\item AB splitting for the  multiscale Schr\"odinger equation,
\item HMM  for the  multiscale Schr\"odinger equation.
\end{itemize}

We apply the initial conditions: \\
$\Psi(x, 0) = \exp(-0.5 (\frac{x- x_c}{s})^2)  cos(\frac{2 \pi (x-x_c)}{\lambda}) + i \; \exp(-0.5 (\frac{x- x_c}{s})^2)  sin(\frac{2 \pi (x-x_c)}{\lambda}) $, where $x \in [0, L]$, where we have $L= 4 \; 10^{-9} \; [m]$, $\lambda = L / 40$ and $s = L/25$.

Further, we assume $V(x,t)=0$ and $V(x,t) = e \;x^2$, \\
where $e =1.6021766 \; 10^{-19} \; [C]$ is the charge of the electron.
We apply for all the methods implicit time-discretization methods
and restrict us to the time-step based on the smallest time-scale.
We apply therefore $\Delta x = 10^{-12}$ and $\Delta t = 10^{-20}$.

In Figure \ref{insta}, we have the numerical results of the
multiscale Schr\"odinger equation with finite difference scheme
(without splitting). This method, we apply for a reference solution.
\begin{figure}[ht]
\begin{center}  
\includegraphics[width=13.0cm,angle=-0]{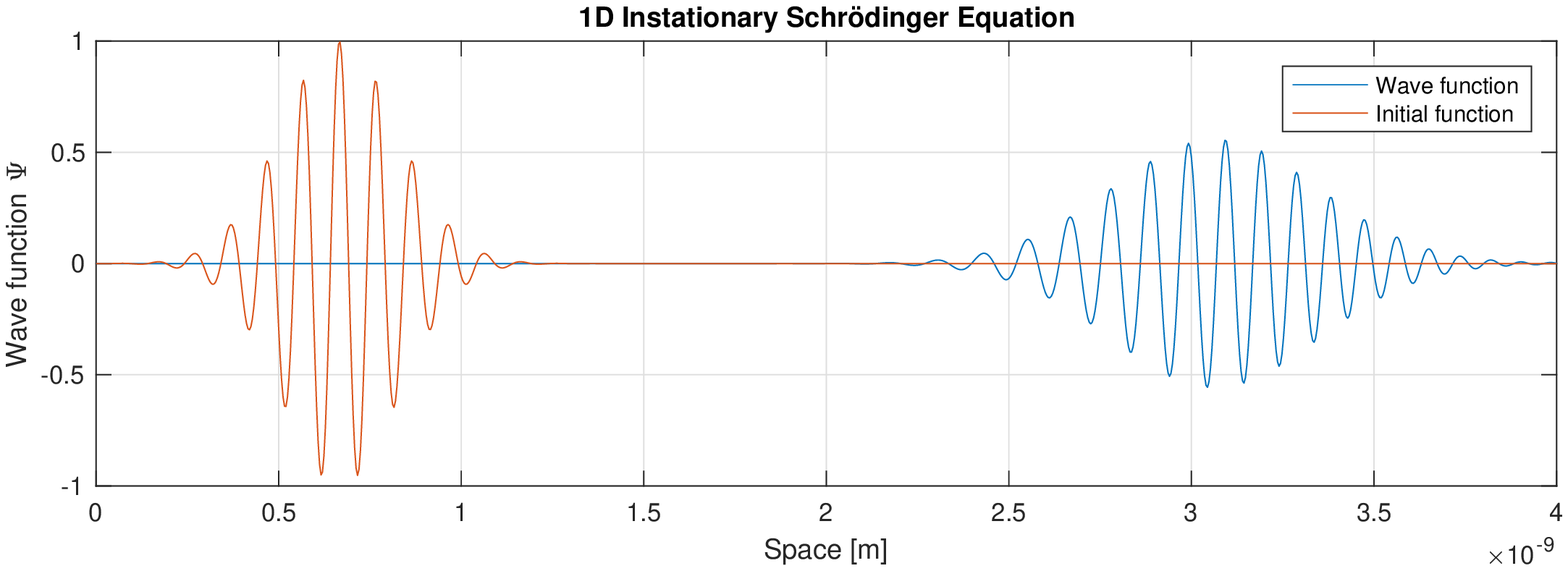}
\includegraphics[width=13.0cm,angle=-0]{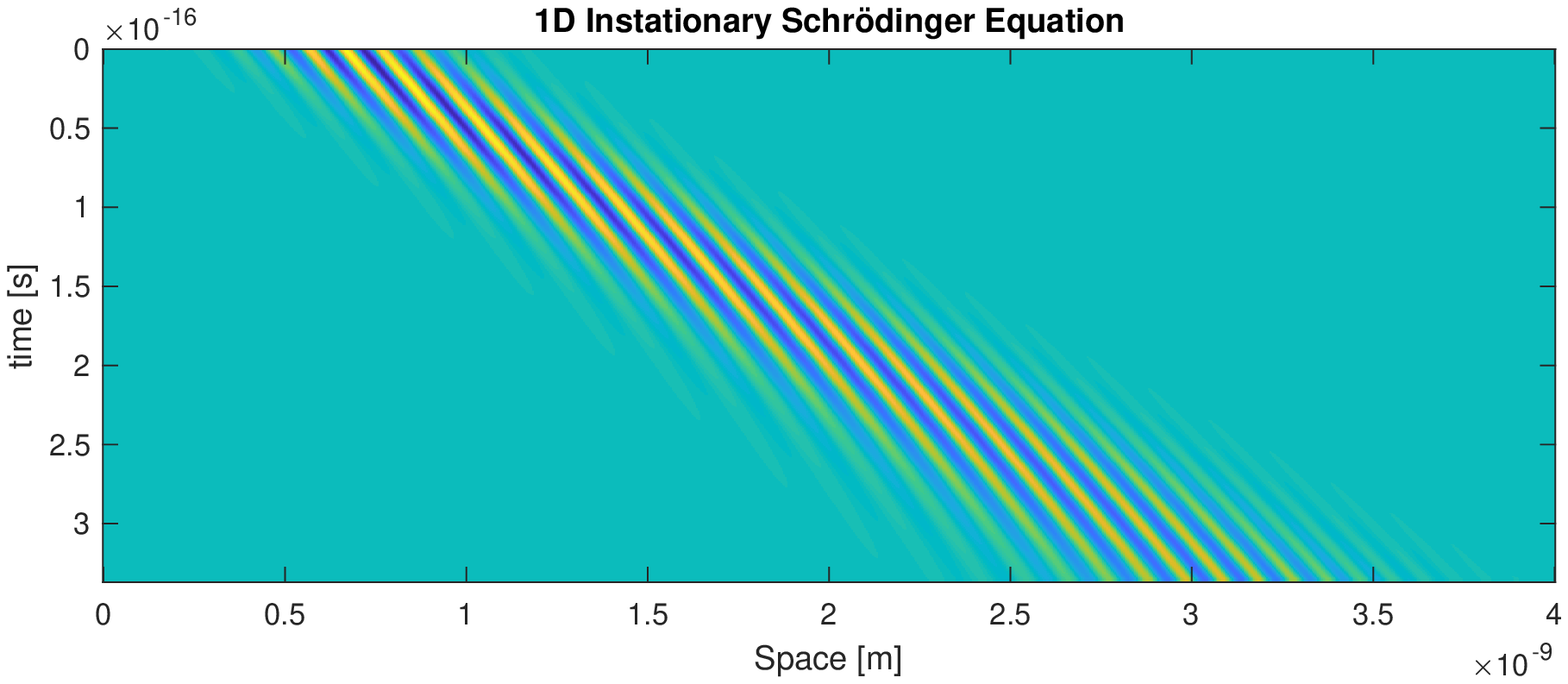}
\end{center}
\caption{\label{insta} Simulations with the instationary Schr\"odinger equation
with finite difference method (without splitting).}
\end{figure}

\begin{remark}
All the numerical methods are stable but we are restricted based on the
small time-steps of the microscopic scale.
Therefore also modified multiscale methods need to much computational
amount such that we tested in a next experiment more appropriate
time-discretization and solver methods.
\end{remark}

\subsection{Test example 2}

We apply the multiscale Schr\"odinger equation, which is given as:
\begin{align}
\label{diff_1_0}
i \epsilon \frac{\partial \Psi(x,t)}{\partial t} &= - \epsilon^2 \frac{\partial^2 \Psi(x,t)}{\partial x^2} + V(x) \Psi(x,t) .
\end{align}

For the studying the different methods, we apply the following error:
\begin{eqnarray}
err_{Method}(\Delta t)&& = norm ( u_{exact}(T) - u_{Method}(T) ) = \\
&& = \sum_{j=1}^J \Delta x \; |u_{exact}(x_j, T) -u_{Method}(x_j, T)| ,
\end{eqnarray}
where $\Delta t$ is the time step, $J$ is the number of spatial steps.
Further $u_{exact}(t_n)$ is the solution with the FD
scheme (\ref{schroed_1_5}) and $u_{Method}(t_n)$ is the solution of the
different methods, means $Method = \{AB, HMM-AB, Extra-AB \}$, see the Algorithms.

The numerical results of the AB-splitting method for the
multiscale Schr\"odinger equation is given in Figure \ref{AB_splitt}.
\begin{figure}[ht]
\begin{center}  
\includegraphics[width=10.0cm,angle=-0]{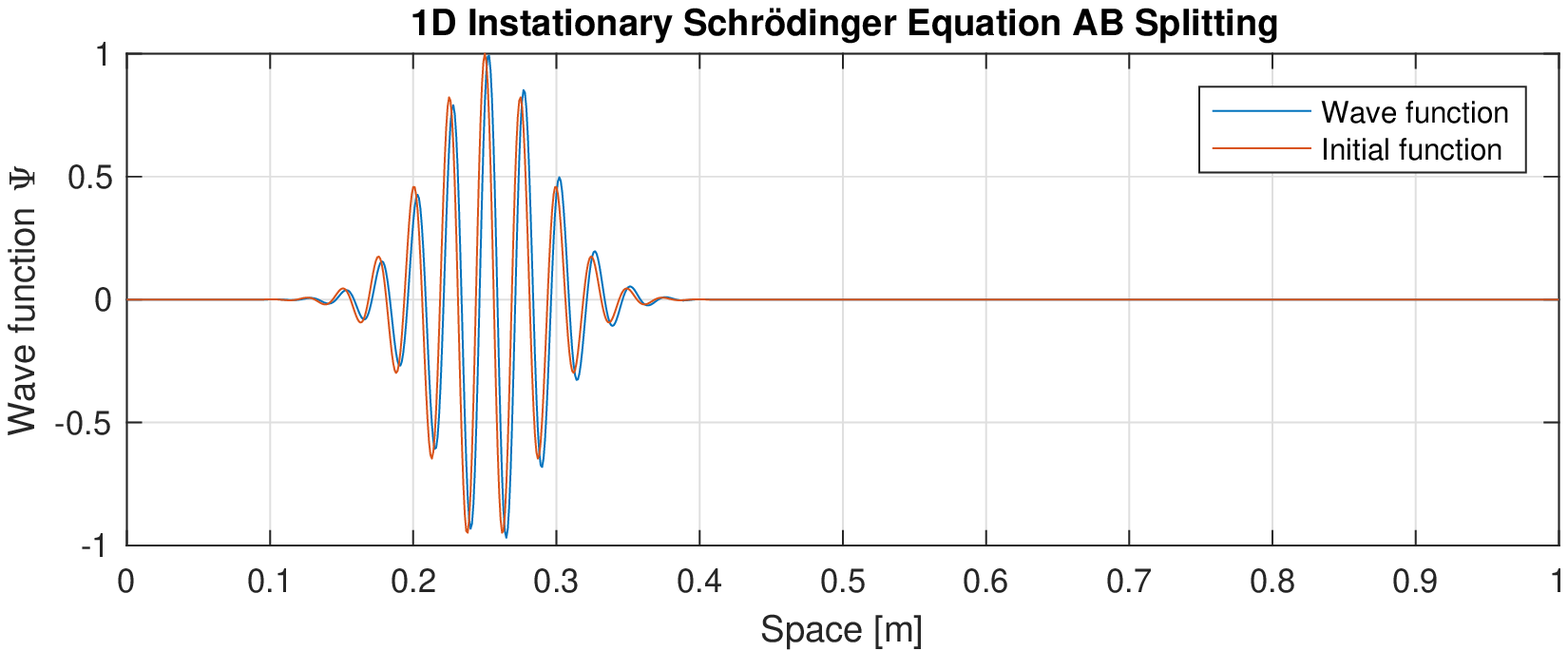}
\includegraphics[width=10.0cm,angle=-0]{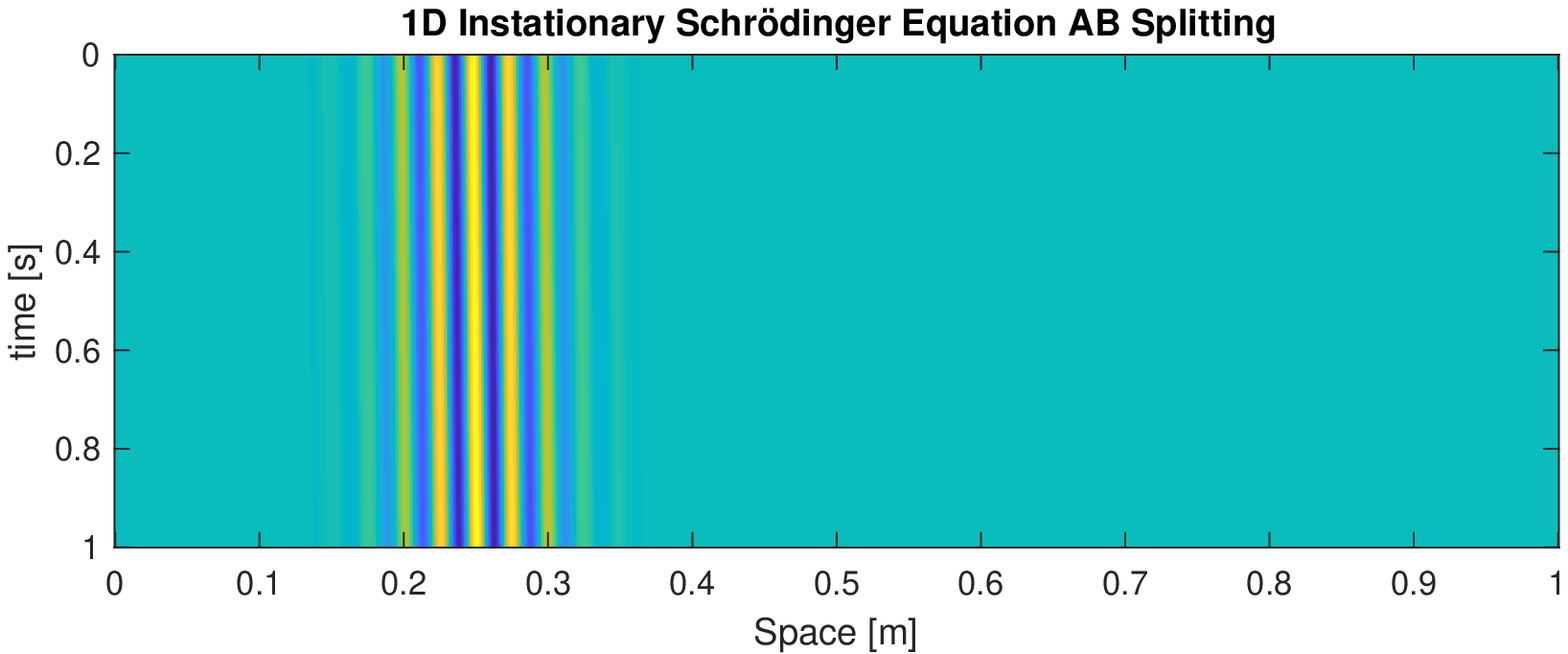}
\end{center}
\caption{\label{AB_splitt} Simulations with the multiscale Schr\"odinger equation with AB-splitting method.}
\end{figure}

\begin{figure}[ht]
\begin{center}  
\includegraphics[width=10.0cm,angle=-0]{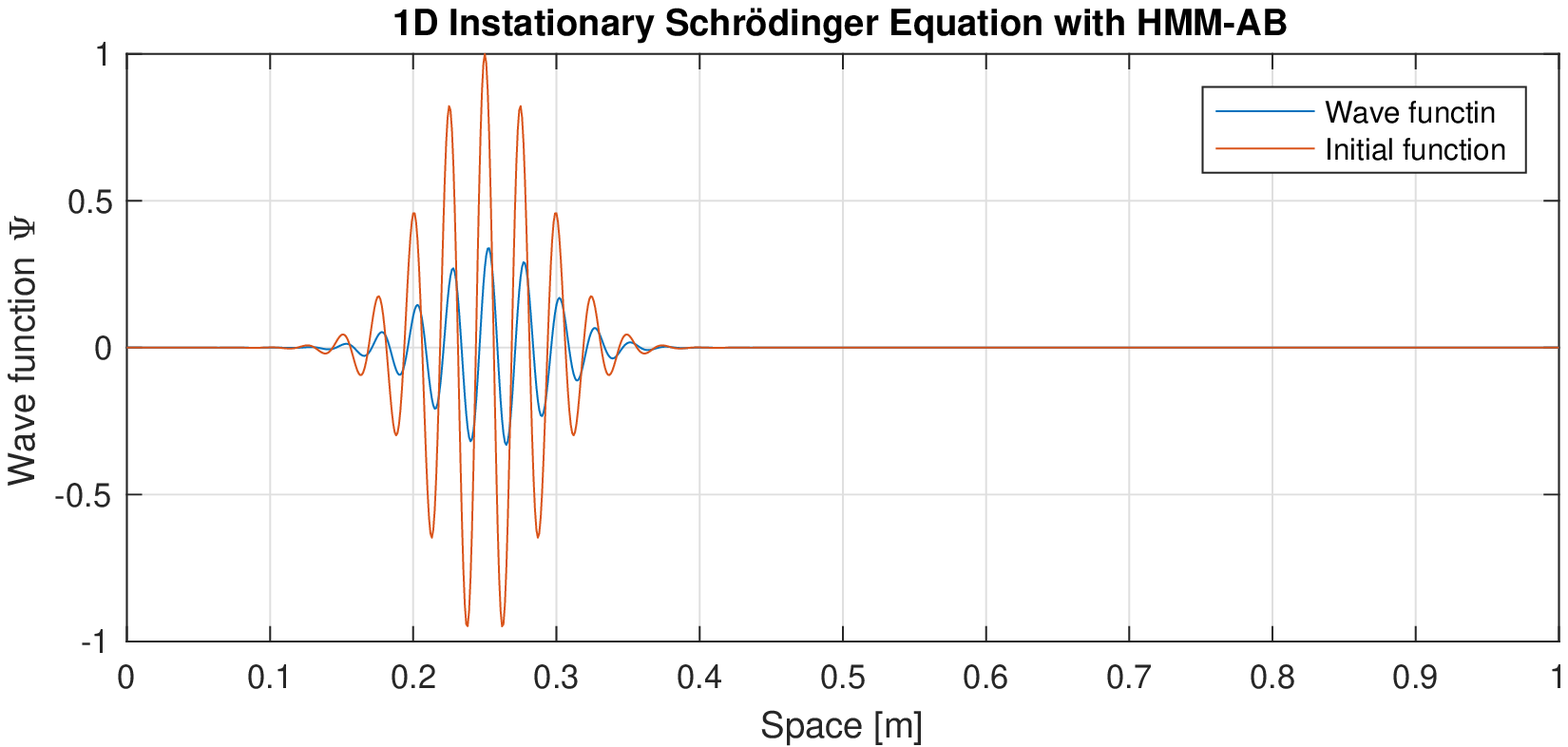}
\includegraphics[width=10.0cm,angle=-0]{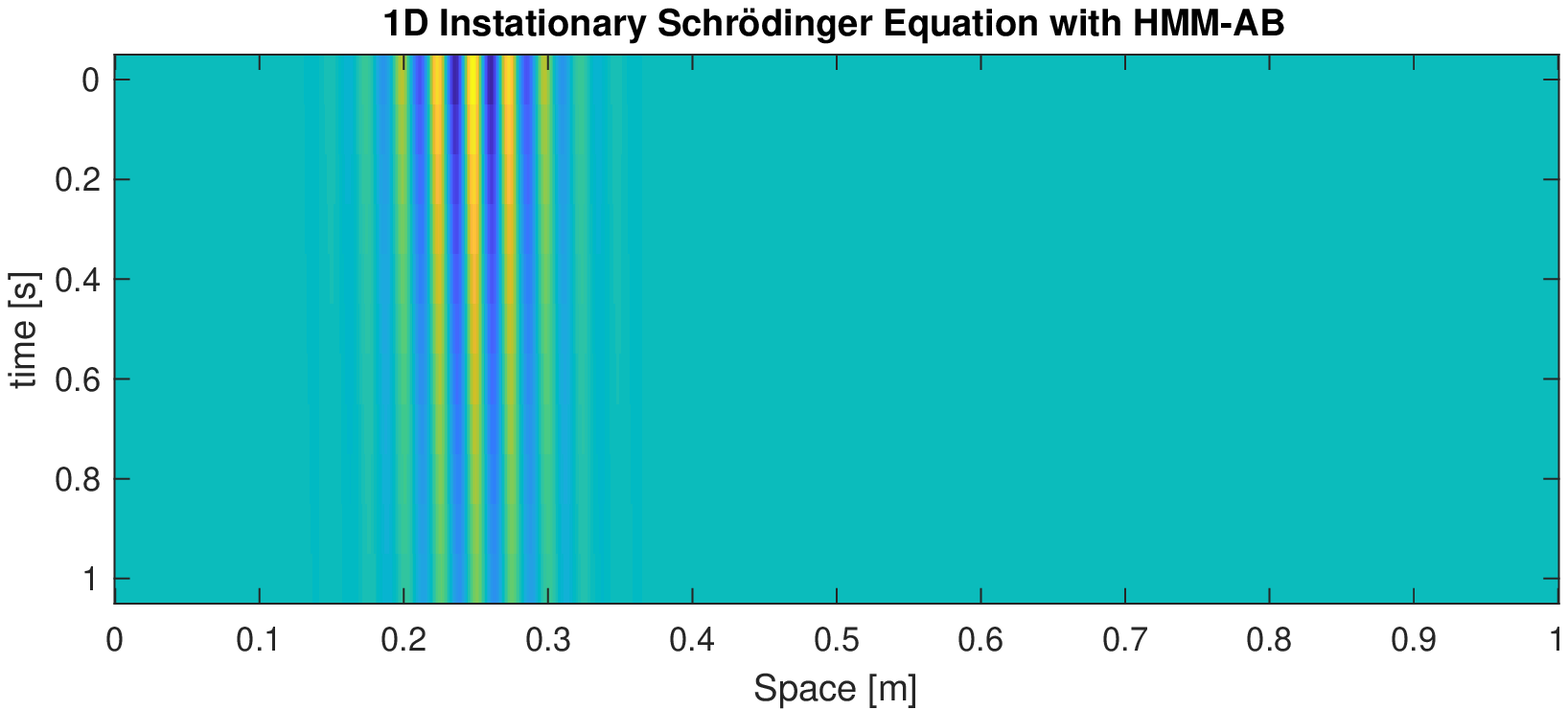}
\end{center}
\caption{\label{HMM} Simulations with the multiscale Schr\"odinger equation with HMM.}
\end{figure}

In the following, we have the error of the different schemes in Table \ref{error_1}.

\begin{table}[ht]
\centering
\begin{tabular}{| c | c | c | c |}
\hline
Numerical method & Numerical error & M & $\tilde{M}$  \\
\hline \hline
AB Splitting & $6.081976027577089e-15$ & $1$ & \\
AB Splitting & $6.081976027555089e-15$ & $10$ & \\
AB Splitting & $6.081976027555089e-15$ & $100$ & \\
\hline
ABA Splitting & $5.918971525959246e-15$ & $1$ & \\
ABA Splitting & $5.918971525959246e-15$ & $10$ & \\
ABA Splitting & $5.918971525959246e-15$ & $100$ & \\
\hline
HMM-AB & $6.081976027555089e-15$ & $1$ & $1$ \\
HMM-AB & $4.404021547293822e-12$ & $10$ & $5$ \\
HMM-AB & Takes too long time & $100$ & $50$ \\
\hline
Extra-AB & $6.081976027555089e-15$ & $1$ & $1$\\
Extra-AB & $5.143436597696588e-15$ & $10$ & $5$ \\
Extra-AB & $5.018633361834493e-15$ & $100$ & $50$ \\
\hline
HigherExtra-AB & $6.081976027555089e-15$ & $1$ & $1$\\
HigherExtra-AB & $7.057484576752767e-15$ & $10$ & $5$ \\
HigherExtra-AB & $9.995071997473121e-15$ & $100$ & $50$ \\
\hline
\end{tabular}
\caption{\label{error_1} The numerical errors of the multiscale methods.}
\end{table}

Further, we present the computational time of the different methods
in Table \ref{comp-time}:
\begin{table}[ht]
\centering
\begin{tabular}{| c | c | c |}
\hline
Numerical Method & Computational Time in sec & Time-step \\
\hline \hline
FD-Scheme & $4.691497$ & $0.0008$  \\
FD-Scheme & $8.785431$ & $0.0004$  \\
FD-Scheme & $18.329370$ & $0.0002$  \\
\hline
AB-Splitting & $4.560601$ & $0.0008$  \\
AB-Splitting & $9.263053$ & $0.0004$  \\
AB-Splitting & $18.514509$ & $0.0002$  \\
\hline
ABA-Splitting & $6.657831$ & $0.0008$  \\
ABA-Splitting & $13.508871$ & $0.0004$  \\
ABA-Splitting & $27.015384$ & $0.0002$  \\
\hline
BAB-Splitting & $6.731909$ & $0.0008$  \\
BAB-Splitting & $13.424463$ & $0.0004$  \\
BAB-Splitting & $28.469047$ & $0.0002$  \\
\hline
Extra-AB & $5.464744$ & $0.0008$  \\
Extra-AB & $10.982060$ & $0.0004$  \\
Extra-AB & $23.133747$ & $0.0002$  \\
\hline
HigherExtra-AB & $7.731819$ & $0.0008$  \\
HigherExtra-AB & $16.434144$ & $0.0004$  \\
HigherExtra-AB & $32.633806$ & $0.0002$  \\
\hline
\end{tabular}
\caption{\label{comp-time} The computational time of the different methods.}
\end{table}

\begin{remark}
Based on the separation of the macroscopic and microscopic operator, we can apply much more adapted time-steps. Further, we save more computational time to
reduce the large amount of microscopic time-steps with extrapolation methods.
The best results are obtained with the extrapolated AB splitting method,
while we separate the operators and need only a smaller number of
microscopic time-steps. Based on higher extrapolation schemes with 
ABA-splitting approaches, we could also improve the accuracy. 
\end{remark}

\section{Conclusion}
\label{concl}

We presented a novel multiscale method, which is based on extrapolation methods and operator splitting approaches. 
Based on the separation of the microscopic and macroscopic operator, we
could reduce the computational time. Further, we applied sufficient microscopic time-steps and extrapolate to the macroscopic time-step, which also 
reduce the computational amount.
Based on the reduced numbers of finer time-steps, we could achieve faster numerical results with the same accurate results as for the standard AB-splitting scheme. Such novel schemes allow to flexiblise the 
standard operator splitting methods and modify such schemes to multiscale methods In future, we will test the new splitting approaches to higher dimensional Schr\"odinger equations and present the numerical analysis of the different schemes.

\bibliographystyle{plain}

\begin{thebibliography}{10}


\bibitem{brez1994}
C.~Brezinski and M.~Redivo-Zaglia.
\newblock{\em Extrapolation methods}.
\newblock Applied Numerical Mathematics, 15(2): 123-131, 1994.

\bibitem{butcher2008}
J.C.~Butcher.
\newblock{\em Numerical methods for ordinary differential equations.}
\newblock 2nd edition, John Wiley \& Sons, Inc. Hoboken, New Jersey, USA, 2008.

\bibitem{we_2007}
W.~E., B.~Engquist, X.~Li, W.~Ren, and E.~Vanden-Eijnden.
\newblock{\em Heterogeneous Multiscale Methods: A Review.}
\newblock Communications in Computational Physics, 2(3), 367-450, 2007.


\bibitem{fargei05}
I.~Farago and J.~Geiser.
\newblock {\em Iterative Operator-Splitting methods for Linear Problems.}
\newblock International Journal of Computational Science and Engineering, 3(4), 255-263, 2007.

\bibitem{geiser_2009_1}
J.~Geiser.
\newblock{\em Decomposition Methods for Partial Differential Equations: Theory
and Applications in Multiphysics Problems.}
\newblock Numerical Analysis and Scientific Computing Series, Taylor \& Francis Group, Boca Raton, London, New York, 2009.

\bibitem{geiser2011}
J.~Geiser.
\newblock{\em Iterative Splitting Methods for Differential Equations}.
Chapman \& Hall/CRC Numerical Analysis and Scientific Computing Series, edited by Magoules and Lai, 2011. 

\bibitem{geiser_2016}
J.~Geiser.
\newblock{\em Multicomponent and Multiscale Systems: Theory, Methods, and Applications in Engineering}.
\newblock Springer, Cham, Heidelberg, New York, Dordrecht, London, 2016.

\bibitem{geiser2017}
J.~Geiser.
\newblock{\em Iterative splitting method as almost asymptotic symplectic integrator for stochastic nonlinear Schr\"odinger equation}.
\newblock AIP Conference Proceedings 1863, 560005, 2017.

\bibitem{geiser2018}
J.~Geiser.
\newblock{\em Multiscale Modelling and Splitting Approaches for Fluids composed of Coulomb-interacting Particles.}
\newblock Special-Issue: Problems with Multiple Time-scales Mathematical, Journal: Mathematical and Computer Modelling of Dynamical Systems, edited by J.Geiser, Taylor and Francis, Abingdon, UK, accepted January 2018.

\bibitem{griffiths2004}
D.J.~Griffiths.
\newblock{\em Introduction to Quantum Mechanics.}
\newblock 2nd edn, Prentice Hall, Upper Saddle River, NJ, 2004.

\bibitem{hai96}
E.~Hairer and G.~Wanner.
\newblock {\em Solving Ordinary Differential Equations II.}
\newblock SCM, Springer-Verlag, Berlin, Heidelberg, New York, No. 14, 1996.

\bibitem{kev2009}
I.G.~Kevrekidis and G.~Samaey.
\newblock  Equation-free multiscale computation: Algorithms and applications.
\newblock {\em Annual Review in Physical Chemistry} 60:321--344, 2009. 

\bibitem{singh2016}
P.~Singh.
\newblock{\em High accuracy computational methods for the semiclassical Schr\"odinger equation}.
\newblock PhD Thesis, Kingfs College DAMTP, Centre for Mathematical Sciences, University of Cambridge, 2016.

\bibitem{stra68}
G.~Strang.
\newblock {\em On the construction and comparison of difference schemes}.
\newblock SIAM J. Numer. Anal., 5, 506-517, 1968.

\bibitem{tak2008}
L.A.~Takhtajan.
\newblock{\em Quantum mechanics for mathematicians.}
\newblock Graduate Studies in Mathematics 95, Amer. Math. Soc., Providence, Rhode Island, 2008.

\bibitem{trotter59}
H.F.~Trotter.
\newblock{\em On the product of semi-groups of operators.}
\newblock Proceedings of the American Mathematical Society, 10(4), 545-551, 1959.

\end{thebibliography}

\end{document}